\documentclass[a4paper,11pt,fleqn]{article}
\usepackage[english]{babel}
\usepackage{amsfonts}
\usepackage{amsthm}
\usepackage{amsbsy}
\usepackage{amssymb}
\usepackage{graphicx}
\usepackage{eso-pic}
\usepackage{dsfont}
\usepackage{tikz}
\usepackage{xfrac}
\usepackage{float}
\usepackage[bookmarks=true]{hyperref}
\usepackage{bookmark}
\usepackage{amsmath}
\usetikzlibrary{matrix,arrows}
\theoremstyle{definition}
\newtheorem{defin}{Definition}[section]
\newtheorem{oss}[defin]{Remark}
\newtheorem{prop}[defin]{Proposition}

\newtheorem{example}[defin]{Example}
\theoremstyle{plain}
\newtheorem{teo}[defin]{Theorem}

\begin{document}
\title{A geometrically bounding hyperbolic link complement}
\author{L. Slavich}
\date{January 2013}

\begin{center}
{\large \textbf{A GEOMETRICALLY BOUNDING\\ HYPERBOLIC LINK COMPLEMENT}

\vspace{0.4 in}}

\textnormal{LEONE SLAVICH}
\end{center}

\noindent ABSTRACT: A finite-volume hyperbolic $3$-manifold  {\em geometrically} bounds if it is the geodesic boundary of a finite-volume hyperbolic $4$-manifold. We construct here an example of non-compact, finite-volume hyperbolic $3$-manifold that geometrically bounds. The $3$-manifold is the complement of a link with eight components, and its volume is roughly equal to $29.311$.

\section{Introduction}
The problem of understanding which hyperbolic $3$-manifolds bound geometrically hyperbolic $4$-manifolds dates back to work of Long and Reid \cite{longreid}, \cite{longreid1}. This problem is related to physics, in particular with the theory of hyperbolic gravitational instantons, as shown in \cite{gibbons}, \cite{ratcliffetschanz} and \cite{ratcliffetschanz3}.
The first example of closed geometrically bounding $3$-manifold was constructed by Ratcliffe and Tschantz in \cite{ratcliffetschanz}. More recent progress in this problem is due to Kolpakov, Martelli and Tschantz \cite{martelli4}. The setup is the following:

\begin{defin}
Let $\mathcal{X}$ be a complete finite-volume hyperbolic $4$-manifold with one totally geodesic boundary component $M$. We say that $M$ {\em geometrically bounds} the manifold $\mathcal{X}$.
\end{defin}
This definition naturally extends the definition of Long and Reid \cite{longreid}, which applies to compact manifolds, to the non-compact case. In this setup, some of the flat cusp sections of $\mathcal{X}$ will have totally geodesic boundary components corresponding to cusp sections of the bounding manifold $M$. Note that the ambient manifold $\mathcal{X}$ may have closed cusp sections which are disjoint from the geodesic boundary $M$.

Even though it is known that every closed $3$-manifold is the boundary of a compact $4$-manifold, Long and Reid \cite{longreid} show that there are closed hyperbolic $3$-manifolds that do not bound geometrically compact hyperbolic $4$-manifolds, therefore the property of being a geometric boundary is non-trivial.
We prove here the following:
\begin{teo}\label{teoremone}
The exterior of the link in Figure \ref{bounding} is hyperbolic. It  is tessellated by eight regular ideal hyperbolic octahedra, and geometrically bounds a hyperbolic $4$-manifold $\mathcal{X}$ which is tessellated by two regular ideal hyperbolic $24$-cells.
\begin{figure}[ht!]
\centering
\makebox[\textwidth][c]{
\includegraphics[width=\textwidth]{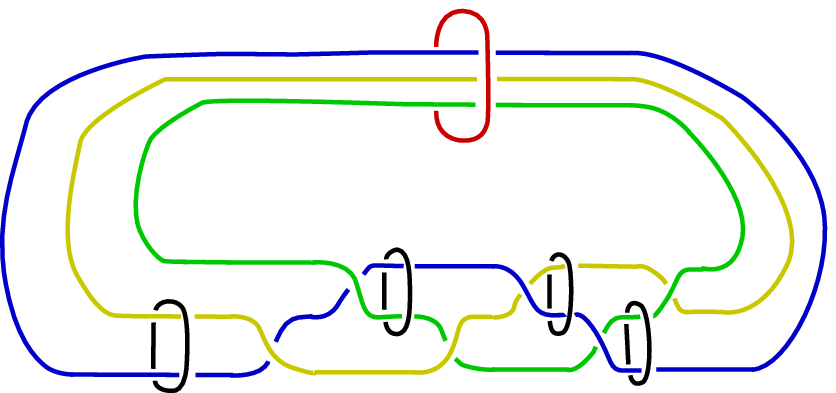}}
\caption{A geometrically bounding hyperbolic link complement in $S^3$.}\label{bounding}
\end{figure}
\end{teo}
The hyperbolic volume of the geometrically bounding $3$-manifold described in Theorem \ref{teoremone} is approximately equal to $29.311$, while the volume of the ambient $4$-manifold $\mathcal{X}$ is equal to $8\pi^2/3$.

The existence of non-compact, complete, finite-volume hyperbolic $4$-manifolds with totally geodesic boundary has been proven by Gromov and Piatetski-Shapiro in \cite{GPS}.
Explicit examples of non-compact, complete, finite-volume hyperbolic $4$-manifolds with connected, totally geodesic boundary are described by Ratcliffe and Tschantz in \cite{ratcliffetschanz}. Moreover, the examples considered there have minimal volume $v_m=4\pi^2/3$. However, it is not clear to the author of the present paper whether or not the geometrically bounding $3$-manifolds are homeomorphic to the exteriors of some link in $S^3$. Notice that all the geometrically bounding manifolds constructed in \cite{ratcliffetschanz} have volume approximaltely equal to $14,655$, \emph{i.e.}\ half the volume of the example of Theorem \ref{teoremone}. 

In Section \ref{ambiente} we will build the ambient manifold $\mathcal{X}$ by taking two copies of the regular ideal hyperbolic $24$-cell, and pairing $40$ of their $48$ total octahedral facets. The remaining $8$ facets will be glued together along their boundaries to produce the bounding manifold $M$. In Section \ref{presentazione} we will use Kirby calculus to build a presentation of $M$ as the exterior of the link of Figure \ref{bounding}.

\medskip

\textbf{Acknowledgements} The author is grateful to his advisor, Prof.\ Bruno Martelli (Universit\`a di Pisa) for the many fruitful discussions and the precious advices that have made this work possible.




\section{Hyperbolic 3-manifolds from triangulations}\label{triang}
In this section we will give a brief overview of a construction contained in \cite{martelli3}, which allows us to build a large family of hyperbolic $3$-manifolds.

Consider the regular euclidean octahedron $O$, which is the convex hull in $\mathbb{R}^3$ of the points $(\pm 1,0,0)$, $(0,\pm 1,0)$ and $(0,0, \pm1)$. 
The octahedron $O$ has eight $2$-dimensional triangular faces, each lying in an affine plane of equation $\pm x_1 \pm x_2 \pm x_3=1$. Notice that these faces have a red/blue checkerboard coloring, so that each edge of $O$ is adjacent to triangles of different colors.

\begin{defin}\label{ottaedroideale}
The octahedron $O$ has a realization as a hyperbolic polytope which is obtained in the following way:
\begin{enumerate}
\item Normalize the the coordinates of the vertices $v_i$, $i=1,\dots,6$ of $O$ so that they lie in the unit sphere $S^{2}\subset \mathbb{R}^{3}$.
\item Interpret $S^{2}$ as the boundary $\partial\mathbb{H}^3$ of hyperbolic space.
\item Consider the convex envelope in $\mathbb{H}^3$ of the points $v_1,\dots,v_6$.
\end{enumerate} 
The convex subset of $\mathbb{H}^3$ that we obtain is the {\em regular ideal hyperbolic} octahedron $\mathcal{O}$. 
\end{defin}
The eight vertices of $O$ correspond to cusps of $\mathcal{O}$, and the checkerboard coloring of the faces of $O$ induces one on $\mathcal{O}$.

\begin{defin}\label{minsky}
The {\em Minsky block} $B$ is the orientable cusped hyperbolic $3$-manifold obtained from two copies $\mathcal{O}_1$ and $\mathcal{O}_2$ of the regular ideal hyperbolic octahedron $\mathcal{O}$ by gluing the {\em red}  boundary faces of $\mathcal{O}_1$ to the corresponding red boundary faces of $\mathcal{O}_2$ with the map which corresponds to the identity on the faces of $\mathcal{O}$.
\end{defin} 
Each of the four boundary components of $B$ is obtained by gluing two copies of a hyperbolic ideal triangle along the boundary via the identity map,  therefore it is a sphere with three punctures. Moreover $B$ has six cusps which all have the same shape, namely a flat annulus of length two and width one (up to homothety).

There is a combinatorial equivalence between certain strata of the octahedron $O$ and those of a regular euclidean tetrahedron $T$.
To see this, notice that it is possible to build an octahedron from a tetrahedron $T$ by truncating $T$ and enlarging the truncated regions until they become tangent at the midpoints of the edges. We can recover the checkerboard coloring, by declaring the truncation faces to be red, and those contained in the faces of $T$ to be blue. As shown in Figure \ref{tetraedrotroncato}, the following correspondences hold:
\begin{enumerate}
\item $\{\mathrm{Vertices\; of}\; T\}\leftrightarrow\{\mathrm{Red\;faces\;of}\; O\}$;
\item $\{\mathrm{Edges\; of}\; T\}\leftrightarrow\{\mathrm{Vertices\;of}\; O\}$;
\item $\{\mathrm{Faces\; of}\; T\}\leftrightarrow\{\mathrm{Blue\;faces\;of}\; O\}$.
\end{enumerate}  

\begin{figure}
\centering
\includegraphics{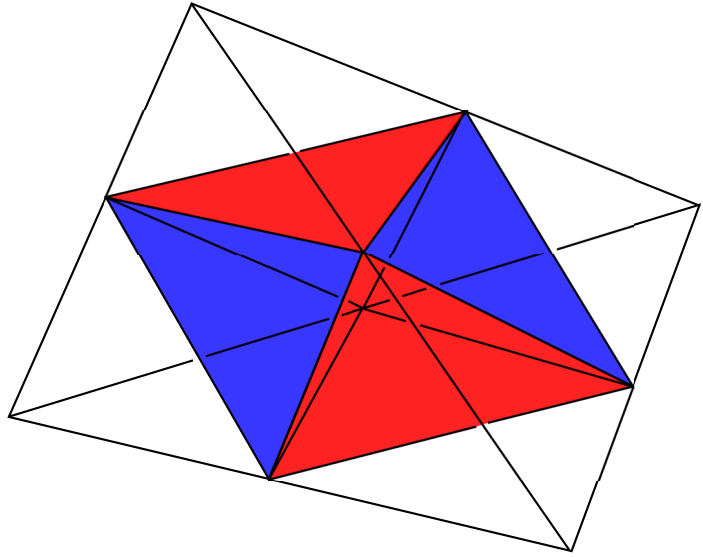}
\caption{Truncating a tetrahedron produces an octahedron.}\label{tetraedrotroncato}
\end{figure}

Since the Minsky block $B$ is the double along the red faces of a regular ideal hyperbolic octahedron $\mathcal{O}$, the correspondences above induce the following correspondences between strata of $B$ and $T$:
\begin{enumerate}
\item $\{\mathrm{Edges\; of}\; T\}\leftrightarrow \{\mathrm{Cusps\; of}\; B\} $.
\item $\{\mathrm{Faces\; of}\; T\}\leftrightarrow \{\mathrm{Boundary\; components\; of}\; B\} $;
\end{enumerate}
\begin{prop}
There is an isomorphism $\phi$ between the symmetry group of the tetrahedron $T$ and the group of orientation-preserving isometries of the block $B$.
\end{prop}
\begin{proof}
Every isometry of $B$ preserves its tessellation into regular ideal hyperbolic octahedra $\mathcal{O}_1$ and $\mathcal{O}_2$.
There is an orientation-reversing involution $j$ of $B$ which exchanges the two ideal octahedra $\mathcal{O}_1$ and $\mathcal{O}_2$, induced by the identity map on $\mathcal{O}$. 
Every orientation-preserving isometry of $B$ can be brought, composing it with $j$ if necessary, to fix the octahedra $\mathcal{O}_i$, $i=1,\dots,2$. This induces a symmetry of the regular octahedron $O$ which preserves the red/blue coloring of its triangular faces. Such a symmetry of $O$ is induced by a unique symmetry of the tetrahedron $T$. This construction defines the map $\phi$, which is easily seen to be an isomorphism.
\end{proof}

A consequence of the correspondences above is that a set of simplicial pairings between the facets of $n$ copies of the tetrahedron $T$ encodes a set of gluings between the boundary components of $n$ copies of the Minsky block $B$, allowing us to produce a hyperbolic $3$-manifold from the simple data of a $3$-dimensional triangulation.

\begin{defin}\label{triangulation}
A {\em $3$-dimensional triangulation} is a pair $$(\{\Delta_i\}_{i=1}^n, \{g_j\}_{j=1}^{2n})$$ where $n$ is a positive natural number, the $\Delta_i$'s are copies of the standard tetrahedron, and the $g_j$'s are a complete set of simplicial pairings between the $4n$ faces of the $\Delta_i$'s.
The triangulation is {\em orientable} if it is possible to choose an orientation for each tetrahedron $\Delta_i$ so that all pairing maps between the faces are orientation-reversing.
\end{defin}

\begin{defin}\label{varietat}
Let  $\mathcal{T}=(\{\Delta_i\}_{i=1}^n, \{g_j\}_{j=1}^{2n})$ be an orientable triangulation as in Definiton \ref{triangulation}. 
We associate to each $\Delta_i$ a copy $B_i$ of the Minsky block $B$. A face pairing $g_j$ between triangular faces $F$ and $G$ of tetrahedra $\Delta_i$ and $\Delta_j$ determines a  {\em unique} orientation-reversing isometry $\phi_j$ between the boundary components of $B_i$ and $B_j$ corresponding to $F$ and $G$, by the condition that their cusps are paired accoding to the pairing of the edges of $\mathcal{T}$. We denote the manifold obtained by pairing the boundary components of the blocks $B_1,\dots,B_n$ via the isometries $\phi_1,\dots,\phi_{2n}$ by $M_{\mathcal{T}}$.
\end{defin}

\begin{prop}
The manifold $M_{\mathcal{T}}$ constructed from a triangulation $\mathcal{T}$ as in Definition \ref{varietat} is an orientable cusped hyperbolic manifold of finite volume.
\end{prop}

\begin{proof}
The manifold $M_{\mathcal{T}}$ is clearly a noncompact manifold with empty boundary. Let us view it as the result of a gluing of regular ideal hyperbolic octahedra. We have a hyperbolic structure on the complement of the $1$-skeleton and we need to check that it extends around the edges. Following chapter $4$ of \cite{notes}, this is equivalent to showing that the horoball sections of the vertex links of the blocks $B_i$, which are flat annuli of length two and width one, glue together to give a closed flat surface. The result of the glueings is to concatenate these annuli along their boundary components, producing a flat torus. 
The volume of $M_{\mathcal{T}}$ is $2n\cdot v_{\mathcal{O}}$, where $n$ is the number of tetrahedra of the triangulation and 
$v_{\mathcal{O}}\approx 3.664$ is the volume of a regular ideal hyperbolic octahedron.
\end{proof}

\begin{oss}
The cusps of the manifold $M_{\mathcal{T}}$ are in one-to-one correspondence with the edges of the triangulation $\mathcal{T}$, and that the number of annuli which are glued together to build a cusp is equal to the valence of the corresponding edge.\end{oss}


Now we will discuss the topological structure of the hyperbolic manifolds constructed from $3$-dimensional triangulations. 
\begin{defin}\label{handlebody}
A {\em relative handlebody} $(H,\Gamma)$ is an orientable handlebody $H$ with a finite sysytem $\Gamma$ of disjoint nontrivial loops in its boundary. A {\em hyperbolic structure} on $(H,\Gamma)$ is a finite-volume complete hyperbolic structure on the manifold $H\setminus \Gamma$, such that the boundary $\partial H \setminus \Gamma$ is totally geodesic.
\end{defin}

Given an orientable $3$-dimensional triangulation $\mathcal{T}=(\{\Delta_i\}_{i=1}^n, \{g_j\}_{j=1}^{2n})$ as in Definition \ref{triangulation}, there is a canonical way to associate to $\mathcal{T}$ a relative handlebody $H_{\mathcal{T}}$. We begin by considering the support $|\mathcal{T}|$ of the triangulation, which is the topological space obtained by glueing the tetrahedra of $\mathcal{T}$ according to the face pairings, with its CW-complex structure. Then we remove a regular open neighborhood of its $1$-skeleton.

We call $H$ the resulting space, which is a handlebody since it has a handle decomposition obtained from a collection of disjoint balls (corresponding to the tetrahedra of $T$) by attaching  $1$-handles  according to the face pairings. The genus of the handlebody is $n+1$, where $n$ is the number of tetrahedra of $\mathcal{T}$.

To each edge $e$ of $\mathcal{T}$ we associate a simple closed loop in $\partial{H}$, corresponding to a simple closed curve in $|T|$ which encircles $e$. These loops form a system $\Gamma$ of curves in $\partial H$, and we define the realtive handlebody $H_T$ as the pair $(H,\Gamma)$.

With a slight abuse of notation, we denote by $H\setminus \Gamma$ the complement in $H$ of the curves of $\Gamma$.
A vital observation is that we can endow $H\setminus \Gamma$ with a hyperbolic structure. To do so, we associate to every tetrahedron of $\mathcal{T}$ a regular ideal hyperbolic octahedron $\mathcal{O}_i$ as in Figure \ref{tetraedrotroncato} and glue the blue faces of the octahedra $\mathcal{O}_i$, $i=1,\dots,n$, together according to the face pairings of $\mathcal{T}$. 

The red faces will glue together along their edges to give the totally geodesic boundary of $H\setminus \Gamma$. Each $\gamma\in \Gamma$ corresponds to an edge $e$ of $\mathcal{T}$, and to each such edge corresponds an equivalence class of ideal vertices of the octahedra. These vertices are glued together to produce the cusp associated to $\gamma$.  
The shape of the cusp associated to the edge $e$ is that of a flat annulus obtained by identyfing sides of $\{0\}\times[0,1]$ and $\{n\}\times[0,1]$ of a rectangle $[0,n]\times[0,1]$, where $n$ is the valence of $e$.

\begin{prop}\label{mirroring}
Given an orientable $3$-dimensional triangulation $\mathcal{T}$, the manifold $M_{\mathcal{T}}$ is the double of $H\setminus \Gamma$ along its boundary, where $H_{\mathcal{T}}=(H,\Gamma)$ is the hyperbolic relative handlebody associated to $\mathcal{T}$.
\end{prop}

\begin{proof}
The manifold $H\setminus \Gamma$ is built by glueing together the blue faces of a set of regular ideal hyperbolic octahedra. Its boundary is tessellated in a loose sense by the red faces of the resulting complex. The manifold $M_{\mathcal{T}}$ is built by glueing together
the boundary components of copies of the block $B$.

Mirroring a regular ideal hyperbolic octahedron $\mathcal{O}$ along its red faces produces the block $B$. The boundary components of $B$ are the doubles of the blue faces of $\mathcal{O}$ along the ideal edges. The glueing of these boundary components is obtained by doubling in the obvious way the glueing of the blue faces of $\mathcal{O}$, since both are determined by the face pairings of  the triangulation $\mathcal{T}$ and the orientation-reversing condition. 
\end{proof}


\subsection{Presentations}
Given a triangulation $\mathcal{T}=(\{\Delta_i\}_{i=1}^n, \{g_j\}_{j=1}^{2n})$, we wish to find a presentation of the associated manifold $M_{\mathcal{T}}$ as a {\em partially framed} link.
\begin{defin}
A partially framed link is the data of a link $L$ in $S^3$, together with a framing on some (not necessarily all) the components of $L$.
\end{defin}
A partially framed link defines a (possibly non-compact) $3$-manifold in the obvious way: we start from the link complement of $L$ and perform a Dehn filling on the framed components, with the coefficients specified by the framing. The components with no framing correspond to non-compact ends of the manifold.

As a consequence of Proposition \ref{mirroring}, the manifold $M_{\mathcal{T}}$ is the complement of a link in a manifold $N$ obtained by mirroring an orientable handlebody $H$ along its boundary.
As mentioned before, the handlebody $H$ consists of $n$ $0$-handles (one for each tetrahedron of $T$), and $2n$ $1$-handles (one for each $g_j$). The handlebody $H$ has genus $n+1$, therefore  the manifold $N$ is a connected sum of $n+1$ copies of $S^2\times S^1$. 
The boundary of $H$ corresponds to an embedded surface $S$ of genus $n+1$ in $N$.

The manifold $N$ has a presentation as surgery on the trivial $n+1$-component link, with all framings equal to zero. This can be easily seen noticing that the Dehn filling on the unknot with zero framing produces $S^2\times S^1$ as a result. 

We can visualize the handlebody $H$ and the surface corresponding to $\partial H$ in this presentation. To do so, we embed the handlebody $H$ in $S^3=\mathbb{R}^3\cup \{\infty\}$, starting from a disjoint union of balls $B_1,\dots,B_n$ in $\mathbb{R}^3$ and connecting them with $1$-handles according to the face pairings. 
We pick a collection of $n+1$ one-handles in such a way that removing them from $H$ yields topologically a $3$-dimensional ball. We encircle these handles with small zero-framed components. The resulting trivial $n+1$ component link, with zero framings on every component, is a presentation of $N$ as mentioned previously. 


Now we have to describe which curves to remove from $N$ to obtain the cusps of $M_{\mathcal{T}}$. Notice that these curves can be represented as a system $A=(a_1, \dots, a_m)$ of disjoint nontrivial loops on the surface $\partial H$ as follows. $\partial H$ has an obvious decomposition into a union of $n$ four-holed spheres $S_1,\dots,S_n$ and $2n$ cylinders $S^1\times I$ connecting their boundary components, induced by the decomposition of $H$ as a union of handles. 

Every $4$-holed sphere $S_i$ corresponds to a tetrahedron of $\Delta_i$ of $\mathcal{T}$, and every boundary component of $S_i$ corresponds to a face of $\Delta_i$. Two of these $4$-holed sphere $S_i$ and $S_j$ share a boundary component if there is a face pairing $g_k$ between the correponding faces of $\Delta_i$ and $\Delta_j$ 
The face pairings of $g_1,\dots,g_{2n}$ define a way to pick in each $S_i$ a complete system of disjoint arcs connecting the boundary components, each arc corresponding  to an edge of $\Delta_i$, in such a way that 
the endpoints of these arcs join along the intersections $S_i\cap S_j$. 
Different choices for these systems of arcs differ by Dehn twists along the curves $S_i\cap S_j$
Adding the loops of $A$ to the presentation of the manifold $N$ produces a presentation of the manifold $M_{\mathcal{T}}$ as a partially framed link.

\begin{oss}
In the construction of this presentation we have made a number of choices: the embedding of the handlebody $H$ in $\mathbb{R}^3$, the choice of the one-handles around which we place the framed components of the link and the number of ``twists'' along the curves $S_i\cap S_j$ we use to connect the arcs to build $A$. All these different choices are related by a sequence of handle slides along zero-framed components.
\end{oss}
\begin{example}\label{semplice}
Consider the orientable triangulation $\mathcal{T}$ obtained by mirroring a tetrahedron $T$ in its boundary. Formally, we take two copies $T_1$ and $T$ of $T$, with opposite orientations, and glue them together along their boundary via the identity map.
A presentation of the manifold $M_{\mathcal{T}}$ associated to this triangulation is given in Figure \ref{clod}.

\begin{figure}[htpb]
\centering
\makebox[\textwidth][c]{
\includegraphics[width=0.4\textwidth]{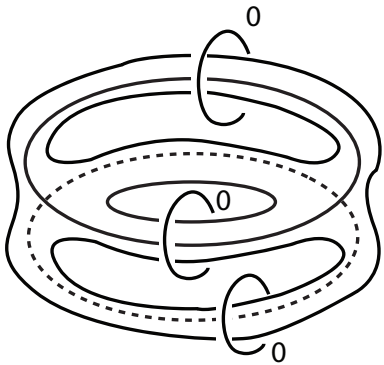}}
\caption{Presentation as a partially framed link of the manifold $M_{\mathcal{T}}$, where $\mathcal{T}$ is the triangulation obtained by mirroring a tetrahedron in its boundary.}\label{clod}
\end{figure}
\end{example}


\section{The ambient $4$-manifold}\label{ambiente}
In this section, we will construct the ambient $4$-manifold $\mathcal{X}$ of Theorem \ref{teoremone}. We begin by describing the $4$-dimensional polytope into which it tessellates.

\subsection{The 24-cell}
The 24-cell $C$ is the only regular polytope in any dimension $n\geq 3$ which is self-dual and not a simplex. It is defined as the convex hull in $\mathbb{R}^4$ of the set of points obtained by permuting the coordinates of $$(\pm 1, \pm 1, 0 , 0 ).$$ 
It has $24$ vertices, $96$ edges, $96$ faces of dimension $2$ and $24$ facets of dimension $3$ which lie in the affine hyperplanes of equations
$$ x_i=\pm 1, \;\;\; \pm x_1 \pm x_2 \pm x_3 \pm x_4 =2. $$
The dual polytope $C^*$ is the convex hull $$C^*=\text{Conv}\{\mathcal{R}, \mathcal{B}, \mathcal{G}\}$$ where
$\mathcal{G}$ is the set of $8$ points obtained by permuting the coordinates of $$(\pm 1,0,0,0)$$ and
$ \mathcal{R}\cup \mathcal{B}$ is the set of $16$ points of the form
$$\left(\pm \frac{1}{2}, \pm \frac{1}{2},\pm \frac{1}{2},\pm \frac{1}{2}\right),$$ 
with $\mathcal{R}$ (resp.\ $\mathcal{B}$) being the set of $8$ points with an even (resp.\ odd) number of minus signs in their entries. 
The facets of $C$ are regular octahedra in canonical one-to-one correspondence with the vertices of $C^*$, and are coloured accordingly in red, green and blue. This coloring is natural, every symmetry of the $24$-cell $C$ preserves the partition of the vertices of $C^*$ into the sets $\mathcal{R}, \mathcal{B}, \mathcal{G}$, and every permutation of $\{\mathcal{R},\mathcal{G},\mathcal{B}\}$ is realized by a symmetry of $C$. 
The vertex figure is a cube, in accordance with self-duality.
Notice furthermore that the convex envelope of $\mathcal{R}\cup \mathcal{B}$ is a hypercube, while the convex envelope of $\mathcal{G}$ is a $16$-cell. 

Being a regular polytope, the $24$-cell has a hyperbolic ideal realization, analogous to the one described in Section \ref{triang} for the case of the octahedron, which we call \emph{hyperbolic ideal $24$-cell} and denote by $\mathcal{C}$.
The facets of $\mathcal{C}$ are regular ideal hyperbolic octahedra.
The vertex figure of the hyperbolic $24$-cell is a euclidean cube, therefore all dihedral angles between facets are equal to $\pi/2$, which is a submultiple of $2\pi$. The $24$-cell is the only regular ideal hyperbolic $4$-dimensional polytope which has this property. This allows us to glue isometrically along their facets a finite number copies of the ideal $24$-cell so that the local geometric structures on each cell piece  together to give a global hyperbolic structure on the resulting non-compact manifold (see \cite{ratcliffetschanz2} and \cite{martelli}).

\subsection{Mirroring the 24-cell}\label{mirror}
We begin the construction of the hyperbolic $4$-manifold $\mathcal{X}$ by mirroring a $24$-cell along its green facets. 
\begin{defin}\label{mirrored}
The {\em mirrored $24$-cell} $\mathcal{S}$ is the space obtained from two copies $\mathcal{C}_1$ and $\mathcal{C}_2$ of the regular ideal hyperbolic $24$-cell $\mathcal{C}$ with opposite orientations, by gluing the green facets of $\mathcal{C}_1$ to the green facets of $\mathcal{C}_2$ via the map which corresponds to the identity on $\mathcal{C}$ .
\end{defin}

\begin{prop}\label{specchiata}
The mirrored $24$-cell $\mathcal{S}$ is a non-compact $4$-dimensional manifold with boundary. The boundary $\partial \mathcal{S}$ decomposes into $16$ strata of dimension $3$, each isomorphic to the Minky block $B$ of Definition \ref{minsky}. These boundary strata intersect at $32$ strata of dimension $2$, each isomorphic to a sphere with three punctures, with dihedral angle $\pi/2$.
\end{prop}
\begin{proof}
Each boundary $3$-stratum is obtained by doubling a red or a blue octahedral facet of $\mathcal{C}$ along the triangular faces that separate it from a green octahedral facet. Up to an appropriate choice of the coloring, this is exactly the construction of Definition \ref{minsky}.
Each boundary $2$-stratum is built by mirroring in its boundary edges a triangular $2$-stratum of $\mathcal{C}$ which separates a red octahedron from a blue one, producing a thrice-punctured sphere.
\end{proof}
The cusp section of the manifold $\mathcal{S}$ is pictured in Figure \ref{cuspide}.
It is obtained by mirroring the cusp section of the $24$-cell, which is a Euclidean cube, along a pair of opposite faces (corresponding to the green facets of $\mathcal{C}$). The result is $Q\times S^1$, where $Q$ is a flat square with sides of length one, and the $S^1$ factor has length two.
The boundary $3$-strata naturally correspond to sets of the form $I\times S^1$, where $I$ is a side of $Q$ and the intersection of any two of these sets in a common edge has indeed angle $\pi/2$.
\begin{figure}[htpb]
\centering
\includegraphics{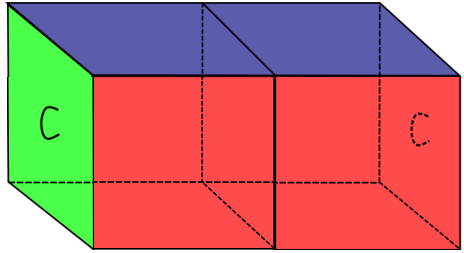}
\caption{A fundamental domain for the cusp section of the mirrored $24$-cell. The opposite faces of each cube share the same colour, and the green faces are identified in pairs.}\label{cuspide}
\end{figure}

There is a natural correpondence between certain strata of the mirrored $24$-cell $\mathcal{S}$ and of the $24$-cell $\mathcal{C}$, described as follows:
\begin{enumerate}
\item $\{\text{Cusps of } \mathcal{S}\} \leftrightarrow \{\text{Vertices of } \mathcal{C}\}$ 
\item $\{\text{Red and blue 3-strata of } \mathcal{S}\} \leftrightarrow \{\text{Red and blue $3$-strata of }\mathcal{C}\}$
\item $\{\text{2-strata of } \mathcal{S}\} \leftrightarrow \{\text{red/blue 2-strata of }\mathcal{C}\}$ 

\noindent By red/blue $2$-stratum we mean a $2$-stratum of $\mathcal{C}$ which bounds a red facet on one side and a blue facet on the other.
\end{enumerate}

The correpondence just mentioned allows us label the strata of $\mathcal{S}$ in the following way:
\begin{enumerate}
\item We label each cusp of $\mathcal{S}$ with the corresponding  vertex of the $24$-cell $\mathcal{C}$, namely by a permutation of a $4$-uple of the form $(\pm1, \pm1,0,0)$, and for brevity we write $\pm$ instead of $\pm 1$. Notice that there always have to be two $0$ entries.
\item Each red or blue facet  of the $24$-cell lies in an affine hyperplane of equations $$\pm x_1 \pm x_2 \pm x_3 \pm x_4 =2$$ and we label the corresponding $3$-stratum of $\mathcal{S}$ by the $4$-uple of $+,-$ signs in this equation. The red/blue coloring is again determined by the parity of the number of minus signs.
\item A red and a blue facet of $\mathcal{C}$ are adjacent along a $2$-stratum if and only if their labels differ by the choice of one sign. Replacing this sign by a $0$, we see that the red/blue $2$-strata of $\mathcal{C}$ and the correponding $2$-strata of $\mathcal{S}$ are naturally labeled by  $4$-uples of $+,-,0$ symbols, with one $0$ entry.
\end{enumerate}
Notice that a $3$-stratum of $\mathcal{S}$ bounds a certain cusp if and only the non-zero entries in the labeling of the cusp coincide with the corresponding $+,-$ entries in the labeling of the $3$-stratum.

There is a group $G$ of affine transformations acting on the $24$-cell and preserving the coloring of the facets. It is generated by the reflections in the hyperplanes $\{x_i=0\}$ and the permutations of the coordinates, and is isomorphic to $(\mathbb{Z}/2\mathbb{Z})^4\times \mathfrak{S}_4$. As a consequence of the correpondence between strata mentioned above, it acts also on the mirrored $24$-cell $\mathcal{S}$, and we can use it to pair the boundary $3$-strata.

\subsection{Face pairings of the boundary 3-strata}
As explained in Proposition \ref{specchiata}, the $2$-strata of the mirrored $24$-cell $\mathcal{S}$ lie at the intersection of a red and a blue $3$-stratum, and the dihedral angle at the intersection is equal to $\pi/2$. As a consequence of this fact, {\em any} pairing of the blue $3$-strata of $\mathcal{S}$ ``kills'' all the $2$-strata, producing a hyperbolic manifold with disjoint totally geodesic boundary components, tessellated by the red $3$-strata of $\mathcal{S}$.

\begin{defin}\label{erre}
We denote by $\mathcal{R}$ the manifold obtained from the mirrored $24$-cell $\mathcal{S}=\mathcal{C}_1\cup \mathcal{C}_2$ by pairing its blue $3$-strata in the following way:
\begin{enumerate}
\item $\pm(+,+,-,+)$ in $\mathcal{C}_1$ (resp. $\mathcal{C}_2$) is paired with $\pm(+,+,+,-)$ in $\mathcal{C}_1$ (resp. $\mathcal{C}_2$) via the map $$F(x,y,z,w)=(x,y,w,z).$$
\item $\pm(+,-,+,+)$ in $\mathcal{C}_1$ (resp. $\mathcal{C}_2$) is paired with $\pm(-,+,+,+)$ in $\mathcal{C}_1$ (resp. $\mathcal{C}_2$) via the map $$G(x,y,z,w)=(y,x,z,w).$$
\end{enumerate}
\end{defin}

\begin{teo}
The manifold $\mathcal{R}$ of Definition \ref{erre} is an orientable hyperbolic non-compact $4$-manifold with totally geodesic boundary.
It has ten cusps whose sections fall into three homothety classes: there are four {\em small} cusps, two {\em medium} cusps and four {\em large} cusps.
\end{teo}

\begin{proof}
The pairing maps  are simply permutation of two coordinates. Since they are orientation reversing, $\mathcal{R}$ is an oriented $4$-manifold with boundary. 
To check that it has a complete hyperbolic structure induced from that of $\mathcal{S}$, we need to check that the cusp sections of $\mathcal{S}$ glue together into Euclidean manifolds to produce the cusp sections of $\mathcal{R}$.

For our purposes, it is convenient to have a graphical representation of the cusp sections of the mirrored $24$-cell $\mathcal{S}$. Consider Figure \ref{cuspide}. Notice that each cusp of $\mathcal{S}$ is bounded by four $3$-strata, two red and two blue. A red and a blue $3$-stratum intesect orthogonally in a $2$-stratum. By taking a vertical section of parallelepiped of Figure \ref{cuspide}, we obtain a planar representation where the cusp corresponds to a square, the boundary $3$-strata correspond to the sides of such square, and the $2$-strata correspond to the vertices as in Figure \ref{sezionecuspide}.

The face pairings will induce identifications of the blue edges of the different squares, and we can describe the cusp shapes of $\mathcal{R}$ by taking the product of the resulting flat surface with a circumference $S^1$ of length two.

\begin{enumerate}
\item The four cusps labeled $\pm(+,+,0,0)$ and $\pm(0,0,+,+)$ are fixed by the face-pairing maps. The opposite blue faces of the corresponding squares are identified by the pairing maps to produce a flat cylinder $C_1$ of width and length equal to one. The cusp shape is a product $C_1\times S^1$. We call these the {\em small} cusps of the manifold $\mathcal{R}$.

\begin{figure}[ht]
\centering
\includegraphics{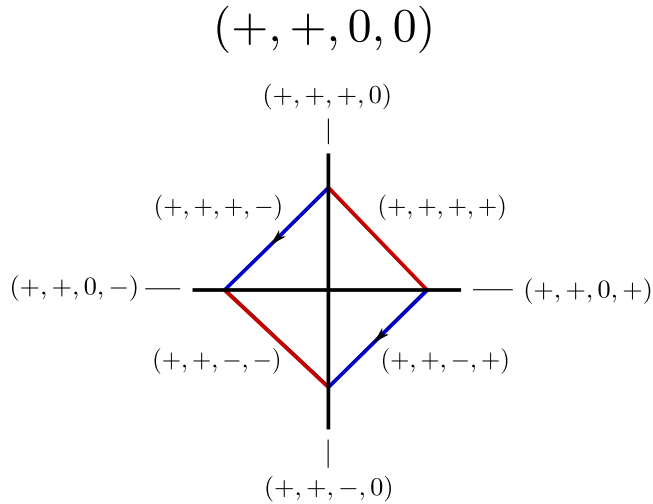}
\caption{Example of planar representation of a cusp section of the mirrored $24$-cell for the small cusp labeled $(+,+,0,0)$. Arrows show the pairings between the blue faces induced by the map $F$.}\label{sezionecuspide}
\end{figure}

\item The cusp $(+,-,0,0)$ is paired to $(-,+,0,0)$ by $G$, and in a similar way $(0,0,+,-)$ is paired to $(0,0,-,+)$ by $F$. The identifications between the first two cusps are represented in the Figure \ref{cuspidemedia}. We call these two the {\em medium} cusps. The cusp shape is a product $C_2\times S^1$, where $C_2$ is a cylinder of width one and  length two.

\begin{figure}
\centering
\makebox[\textwidth][c]{
\includegraphics[width=\textwidth]{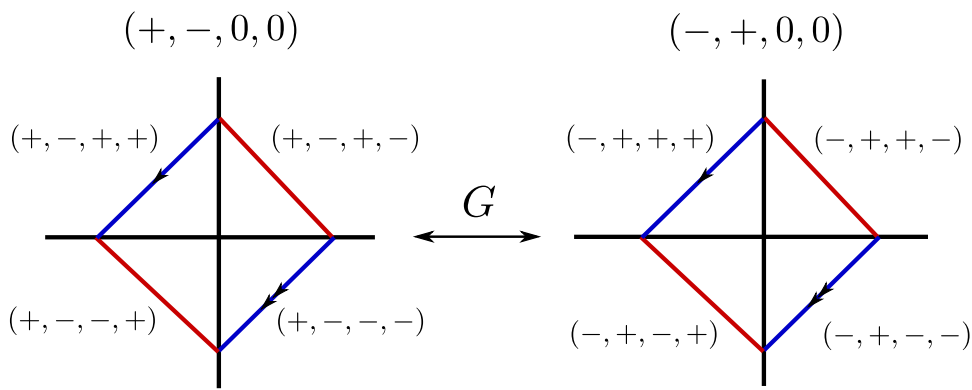}}
\caption{Face pairings for the medium cusp obtained by pairing $(+,-,0,0)$ to $(-,+,0,0)$ via the map $G$.}\label{cuspidemedia}
\end{figure}

\item The remaining sixteen cusps of the mirrored $24$-cell are identified in four groups of four. In all these four cases, the pairings between the blue boundary components produce the same flat manifold, and one example is shown in Figure \ref{cuspidegrande}. We call these cusps {\em large}. The cusp shape is a product $C_3\times S^1$, where $C_3$ is a cylinder of width one and  length four.

\begin{figure}
\centering
\makebox[\textwidth][c]{
\includegraphics[width=\textwidth]{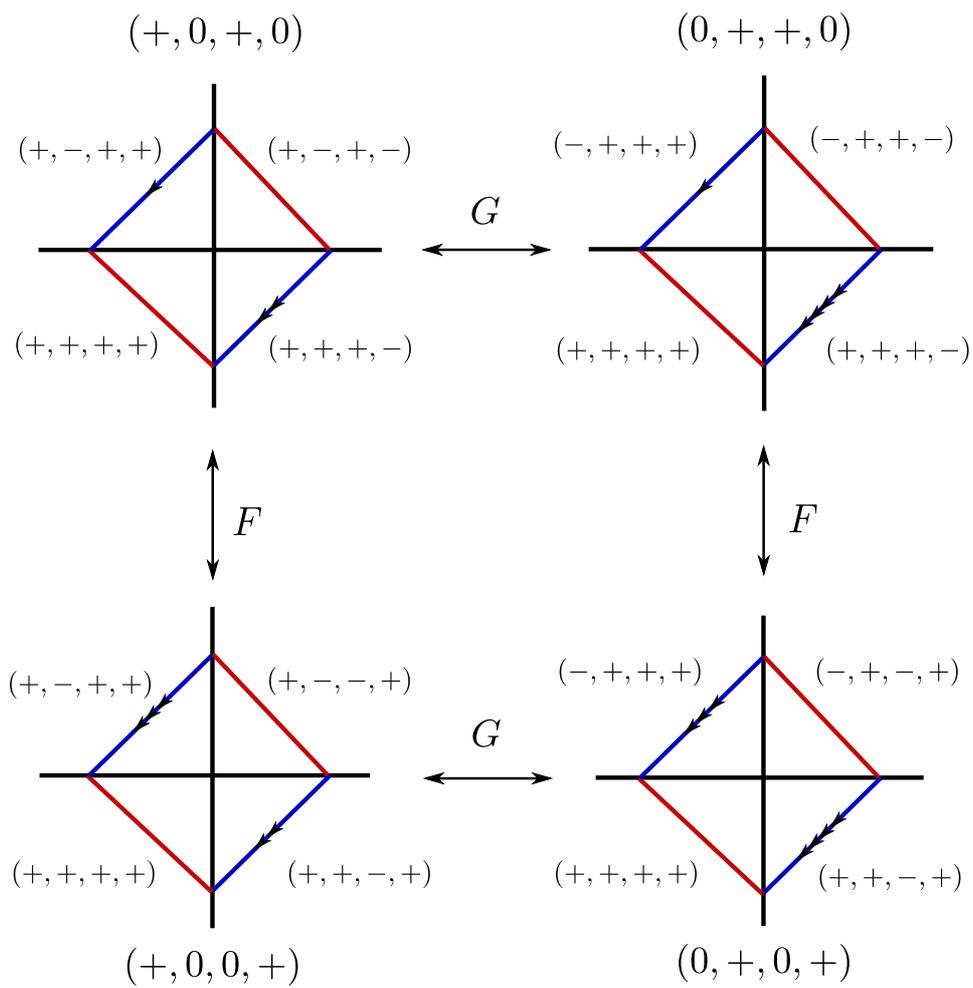}}
\caption{Face pairings for the large cusps.}\label{cuspidegrande}
\end{figure}
\end{enumerate}
\end{proof}

\begin{prop}\label{5bordi}
The manifold $\mathcal{R}$ of Definition \ref{erre} has five disjoint totally geodesic boundary components which fall into two isometry classes: there are four {\em small} boundary components and one {\em large} boundary component.
\end{prop}

\begin{proof}
The boundary components of $\mathcal{R}$ are obtained by glueing together the red $3$-strata of the mirrored $24$-cell $\mathcal{S}$ along their boundary $2$-strata. Notice that the {\em red} boundary octahedra $O_i$ of the $24$-cell have a natural green/blue checkerboard coloring on their faces. Suppose $f$ is a $2$-dimensional face of $O_i$. The face $f$ will be adjacent to $O_i$ and another octahedron $O_j$ which will be colored either in green or blue. We color $f$ with the color of $O_j$.
Each boundary $3$-stratum of $\mathcal{S}$ is obtained by mirroring an octahedron $O_i$ along its green faces, therefore it is isomorphic to the Minsky block $B$ of Definition \ref{minsky}.

The pairings between the blue facets of $\mathcal{S}$ induce glueings between the $2$-dimensional faces of these boundary $3$-strata, producing boundary components whose topology is encoded by triangulations as explained in Section \ref{triang}.

We begin by considering the $3$-strata of $\mathcal{S}$ labeled $\pm(+,+,+,+)$ and $\pm(+,+,-,-)$. For each such $3$-stratum, it is easy to check that the four boundary $2$-strata are identified in pairs. The resulting $3$-manifolds are all isometric, and are obtained from a triangulation $\mathcal{T}$ with one tetrahedron. We call these the {\em small} boundary components of the manifold $\mathcal{R}$.
Let's take a look at the case of the component labeled $(+,+,+,+)$. Its four boundary components are labeled $(+,+,+,0)$, $(+,+,0,+)$, $(+,0,+,+)$ and
$(0,+,+,+)$. The pairings are shown below in Figure \ref{bordopiccolo}, together with their behaviour on the cusps.

\begin{figure}[ht!]
\centering
\makebox[\textwidth][c]{
\includegraphics[width=\textwidth]{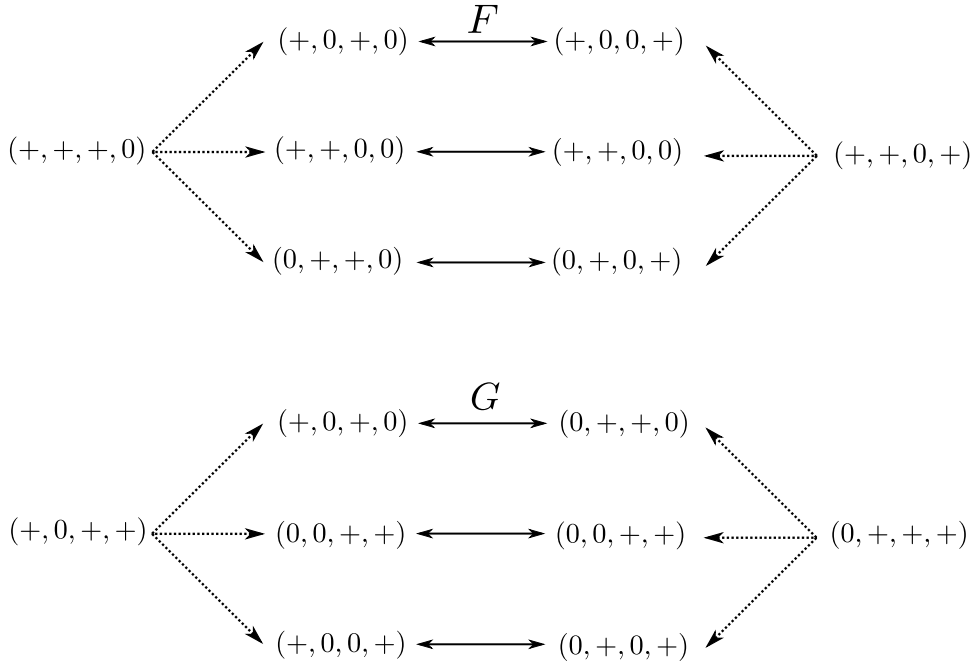}}
\caption{Face pairings for the boundary component $(+,+,+,+)$. Dotted lines indicate adjacencies between boundary $2$-strata and the cusps of $\mathcal{S}$.}\label{bordopiccolo}
\end{figure}

The triangulation associated to the small boundary components is represented in figure \ref{untetraedro}, and is obtained in the following way. Start with one tetrahedron $\mathcal{T}$, and pick a pair of opposite edges $a$ and $b$, such that $a$ separates face $A_1$ and $A_2$ and $b$ separates faces $B_1$ and $B_2$. Identify the faces $A_1$ and $A_2$ by the unique orientation reversing map which fixes their common edge, i.e.\ by ``folding'' along $a$, and do the same for $B_1$ and $B_2$. The manifold defined by such triangulation has three cusps. Two are boundary components of a small cusp of $\mathcal{R}$, and one is a boundary component of a large cusp of $\mathcal{R}$.

\begin{figure}[ht!]
\centering
\includegraphics{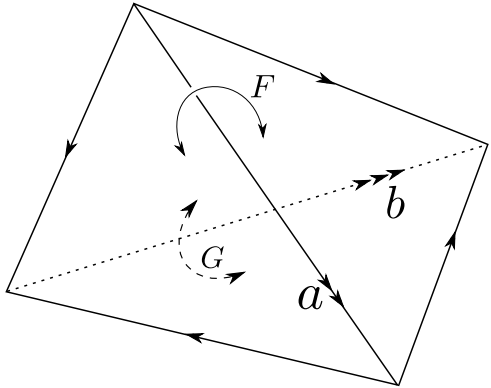}
\caption{Triangulation associated to a small boundary component.}\label{untetraedro}
\end{figure}

There is another {\em large} boundary component $M$, obtained by glueing together the  four faces labeled $\pm(+,-,+,-)$ and $\pm(+,-,-,+)$. The glueings are shown in Figure \ref{bordogrande2}. 
\begin{figure}
\centering
\makebox[\textwidth][c]{
\includegraphics[width=\textwidth]{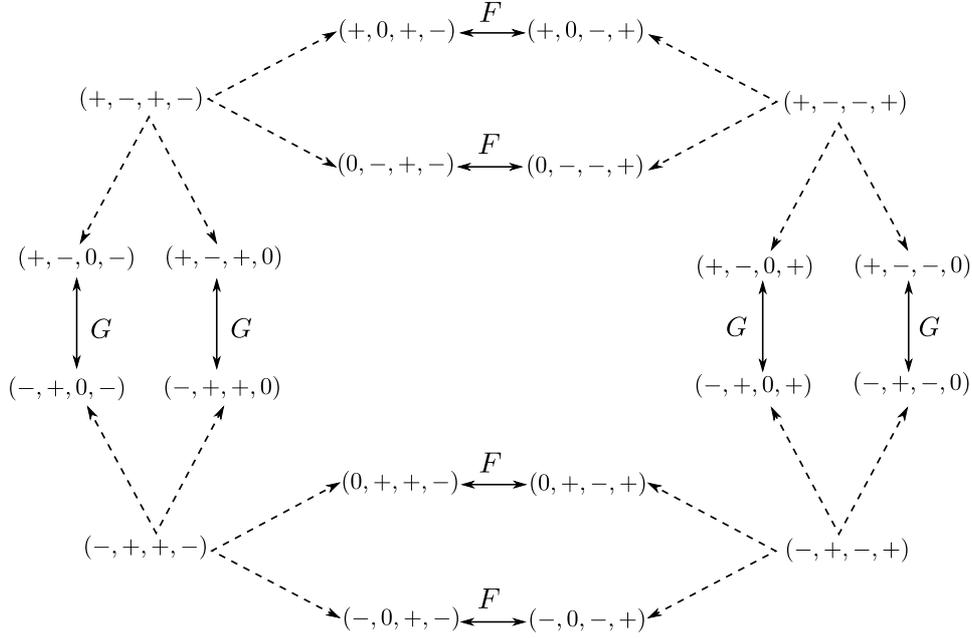}}
\caption{Face pairings for the large boundary component. Dashed lines indicate adjacencies between boundary $3$-strata and their boundary $2$-strata. The arrows labeled $F$ and $G$ show the face pairings between the $2$-strata.}\label{bordogrande2}
\end{figure}

We have to determine how the faces are glued together, and this depends on the behaviour of the pairings on the cusps. To do so, we represent each boundary $3$-stratum $\mathcal{B}_i$ with a copy $\mathcal{T}_i$ of the tetrahedron. Recall that the faces of $\mathcal{T}_i$ are in one-to-one correspondence with the boundary $2$-strata of $\mathcal{B}_i$. We assign to the faces of each tetrahedron a number $n\in\{1,2,3,4\}$, with the rule that to a face is assigned the number $n$ if the label of corresponding boundary $2$-stratum has a zero in the $n$th entry.
In each tetrahedron, every face is assigned a different number, and each edge is identified by the couple $\{i,j\}$ of integers assigned to its two adjacent faces.

Notice that all face pairings identify faces of different tetrahedra $\mathcal{T}_i$ and $\mathcal{T}_j$ in a way which {\em preserves} their numbering. The map $F$ identifies faces numbered with $n\in\{1,2\}$, while $G$ identifies faces numbered with $n\in\{3,4\}$. Furthermore the map $F$ (resp.\ $G$) identifies the edges labeled $\{1,2\}$  (resp.\ $\{3,4\}$), and there is a unique way to do so in an orientation-reversing way.
The resulting triangulation is represented in Figure \ref{bordogrande}.
\begin{figure}[ht]
\centering
\makebox[\textwidth][c]{
\includegraphics[width=0.8\textwidth]{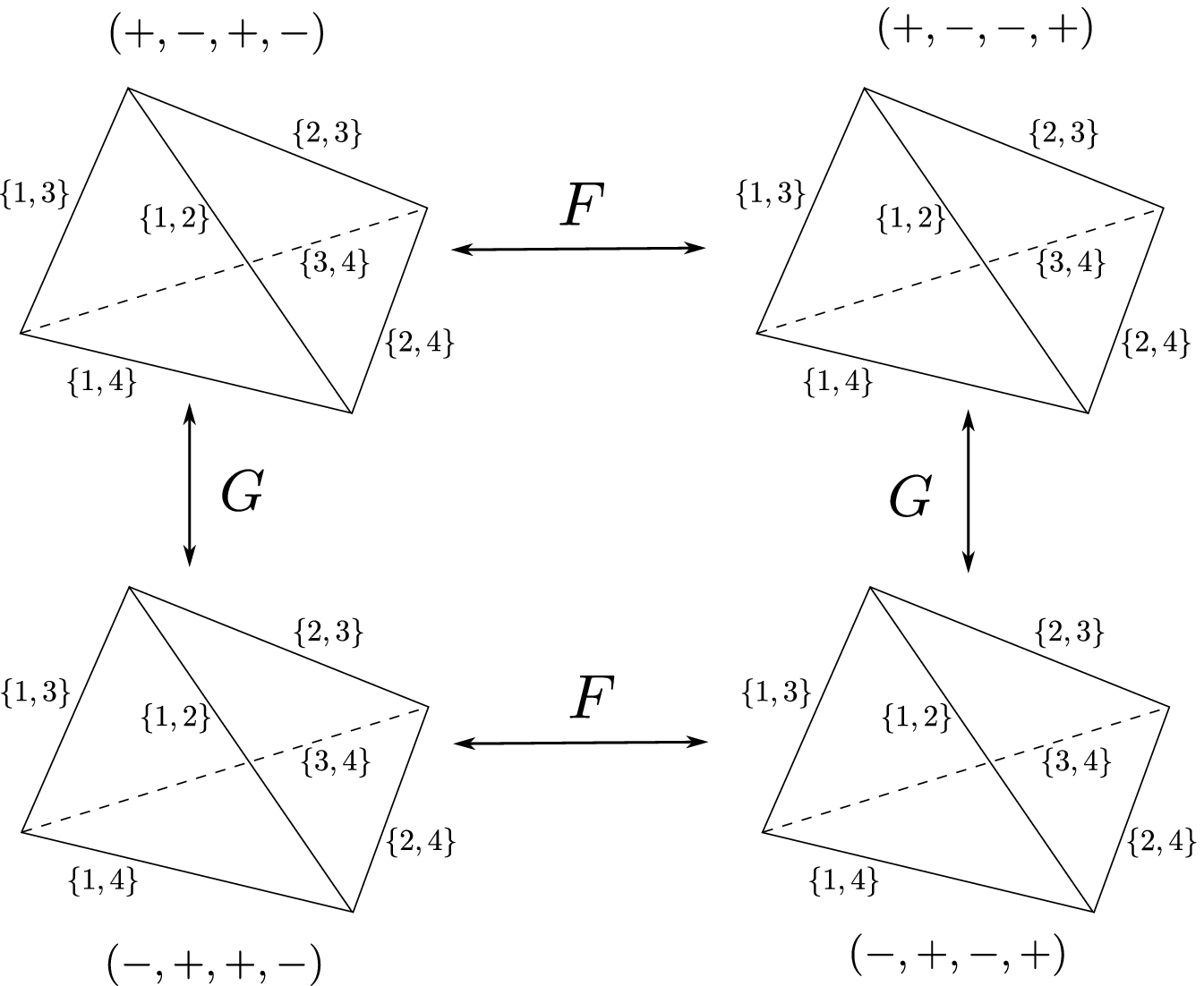}}
\caption{Triangulation encoding the large boundary component $M$ of $\mathcal{X}$. The map $F$ pairs the two top faces of the indicated tetrahedra, sending edges labeled $\{1,2\}$ to 
edges labeled $\{1,2\}$, and edges labeled $\{i,3\}$ to edges labeled $\{i,4\}$ for $i=1,2$. The map $G$ pairs the bottom faces, sending edges labeled $\{3,4\}$ to edges labeled $\{3,4\}$, and edges labeled $\{1,i\}$ to edges labeled $\{2,i\}$ for $i=3,4$. All pairings are orientation reversing.}\label{bordogrande}
\end{figure}

\end{proof}

\begin{oss}
The volume of the large boundary component $M$ of the manifold $\mathcal{R}$ is $8\cdot v_{\mathcal{O}}\approx 29.311$, where $v_{\mathcal{O}}$ is the volume of the regular ideal hyperbolic octahedron.
\end{oss}

\begin{oss}\label{cuspidierre}
The large boundary component $M$ of the manifold $\mathcal{R}$ has eight cusps, four of which come from boundary components of large cusps of $\mathcal{R}$. They have the cusp shape of a torus obtained by identifying opposite sides of a $2\times 4$ rectangle. We call these the {\em large} cusps of $M$.
The other four come from boundary components of the medium cusps of $\mathcal{R}$, and they correspond to the edges of the tetrahedra labeled $\{1,2\}$ and $\{3,4\}$. Their shape is that of a torus obtained by identifying opposite sides of a square of sidelength two. These are the {\em small} cusps of $M$.
\end{oss}

\begin{teo}\label{teoremone2}
The large boundary component $M$ of the manifold $\mathcal{R}$  geometrically bounds a cusped orientable hyperbolic $4$-manifold $\mathcal{X}$ of finite volume.
\end{teo}

\begin{proof}
As stated in Proposition \ref{5bordi}, $\mathcal{R}$ has a total of five disjoint totally geodesic boundary components, four of which (the small ones) are isometric to each other, so it suffices to pair these four components together isometrically. We may for instance consider the orientation reversing map $K:\mathbb{R}^4\rightarrow\mathbb{R}^4$ defined by $$K(x,y,z,w)=(-y,-x,z,w).$$ We use it to pair the boundary components of $\mathcal{R}$ with labels
 $\pm(+,+,+,+)$ and $\pm(+,+,-,-)$, obtaining an orientable $4$-manifold $\mathcal{X}$. The hyperbolic structure on $\mathcal{R}$ extends to one on $\mathcal{X}$, which has one totally geodesic boundary component isomorphic to $M$.
The manifold $\mathcal{X}$ is tessellated by two regular ideal hyperbolic $24$-cells, and has volume $2\cdot v_m$, where $v_m=4\pi^2/3$ is the volume of the regular ideal hyperbolic $24$-cell.
\end{proof}

\begin{oss}
The volume of a hyperbolic $4$-manifold $\mathcal{X}$ with totally geodesic boundary is proportional to its Euler characteristic via the rule 
$$\text{Vol}(\mathcal{M})=4\pi^2/3\cdot \chi(\mathcal{X}).$$ This can be easily seen by noticing that the manifold obtained by doubling $\mathcal{X}$ in its boundary has twice the Euler characteristic of $\mathcal{X}$, since the Euler characteristic of the hyperbolic boundary of $\mathcal{X}$ is zero. Moreover, it has twice the volume of $\mathcal{X}$. The claim then follows from applying the generalized Gauss-Bonnet theorem.

This allows us to conclude that the manifold $\mathcal{X}$ constructed in the proof of Theorem \ref{teoremone2} has Euler characteristic two, or equivalently twice the mi\-nimal volume for a hyperbolic manifold with totally geodesic boundary. 
\end{oss}

\section{The boundary manifold as a link complement}\label{presentazione}
In this section we will exhibit a presentation of the geometrically bounding manifold $M$ as link complement in $S^3$.
Recall from Section \ref{triang} that it is possible to associate a presentation as a partially framed link to every orientable manifold built from a triangulation $\mathcal{T}$. We have seen in the previous chapter that the geometrically bounding manifold $\mathcal{M}$ is constructed from a triangulation $\mathcal{T}$ with four tetrahedra
as in figure \ref{bordogrande}.

The presentation of $M$ as partially framed link $L$ is shown in Figure \ref{framedlink1} (top). The zero-framed components are labeled $F_i$, for  $i=1,\dots,5$. 
The four components labeled $y,r,g,b$ correspond to the large cusps of $M$, while the components labeled $s_i$, for $i=1\dots,4$ correspond to the small ones. 

\begin{figure}[ht!]
\centering
\makebox[\textwidth][c]{
\includegraphics[width=\textwidth]{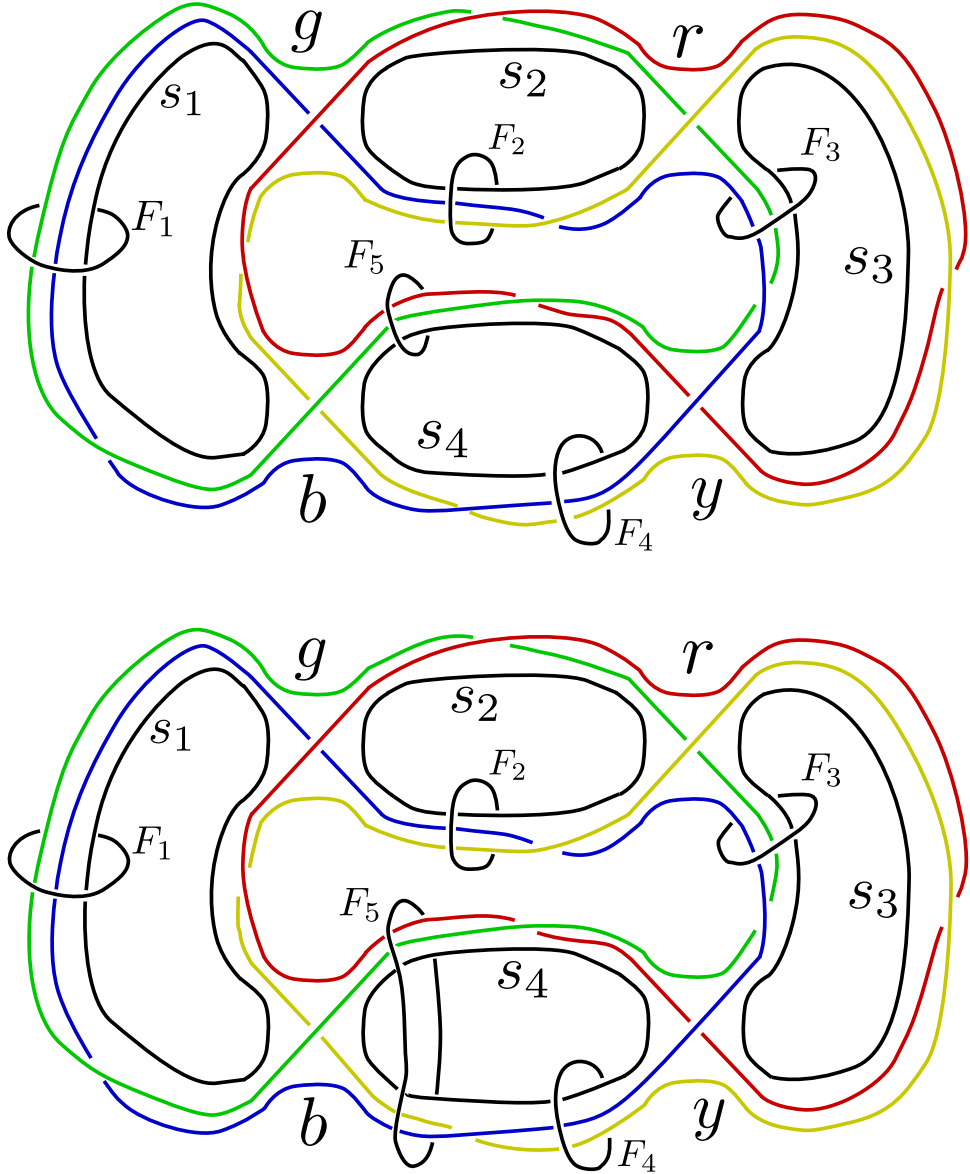}}
\caption{Presentations of the geometrically bounding manifold $M$ as a partially framed link.}\label{framedlink1}
\end{figure}

We wish to eliminate the framed components $F_i$ by applying blow-ups/downs and handle-slides. We begin by applying a handle-slide of $F_5$ over $F_4$, modifying $L$ to the link of Figure \ref{framedlink1} (bottom). For simplicity, we keep calling the $F_5$ new framed component. Notice that the component labeled $r$ is unlinked from all the framed components except $F_5$.

As a second step, we get rid of the framed components $F_i$, for $i=1,\dots,4$. We use the local move represented in Figure \ref{mossalocale}. 
This move can be viewed as a composition of two moves. The first is a twist on an unknotted non-framed component, which does not change the underlying manifold as shown in page $265$ of \cite{rolfsen}. The second is a topological blow-down.

\begin{figure}[ht!]
\centering
\includegraphics{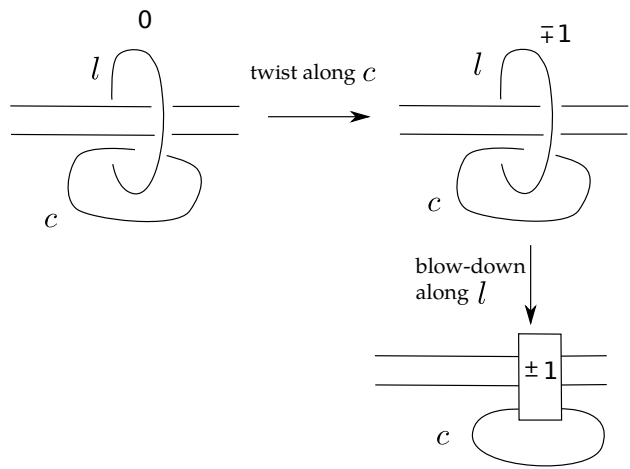}
\caption{Local twist along an unknotted zero-framed component $l$. This move is the result of a twist along an unknotted, non-framed component $c$, which changes to $\mp 1$ the framing on $l$, followed by a blow-down on $l$. The rectangle labeled by $\pm 1$ represents a full twist of the strands that cross it.}\label{mossalocale}
\end{figure}

We apply this local move within small balls that encircle the unknotted framed components $F_i$, $i=0,\dots,4$ and the unknotted non-framed components $s_i$, $i=1,\dots, 4$. The role of $l$ is played by the components labeled $F_i$, while the role of $c$ is played by the components labeled $s_i$. We choose the directions of the twist in such a way that that clockwise (resp.\ counterclockwise) half-twists of the coloured components become counterclockwise (resp.\ clockwise). In other words we choose the twists in order to minimize the linking number of the coloured components.

The result of these moves is shown in Figure \ref{framedlink3}. There is just one component, $F_5$, with zero framing, so what we have now is a presentation of the manifold $M$ as a link complement in $S^2\times S^1$. In fact the components labeled $y,r,g,b$ form a pure braid in a solid torus $D^2\times S^1 \subset S^2\times S^1$, with the loops $s_1,s_2,s_3,s_4$ encircling components labeled $y,r,b$.

\begin{figure}[ht!]
\centering
\includegraphics{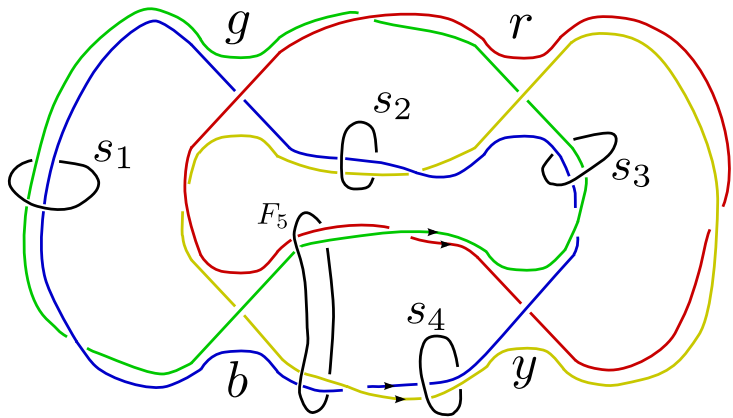}
\caption{Presentation of the geometrically bounding manifold $M$ as a link complement in $S^2\times S^1$.}\label{framedlink3}
\end{figure}

The braid formed by the components labeled $y,r,g,b$ is represented in Figure \ref{braid}. The braid labeled $r$ winds once around the sub-braid formed by the components $y,g,b$. As a consequence of this fact, we see that the framed link of Figure \ref{framedlink3} is isotopic to the framed link of Figure \ref{framedlink4}.

\begin{figure}[ht!]
\centering
\includegraphics{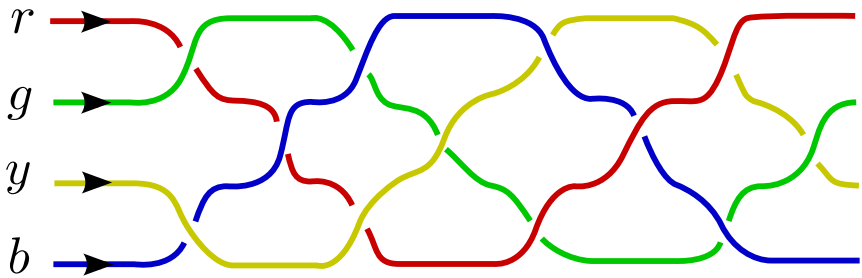}
\caption{Pure braid in $S^2\times S^1$ formed by the coloured components.}\label{braid}
\end{figure}

\begin{figure}[ht!]
\centering
\makebox[\textwidth][c]{
\includegraphics[width=\textwidth]{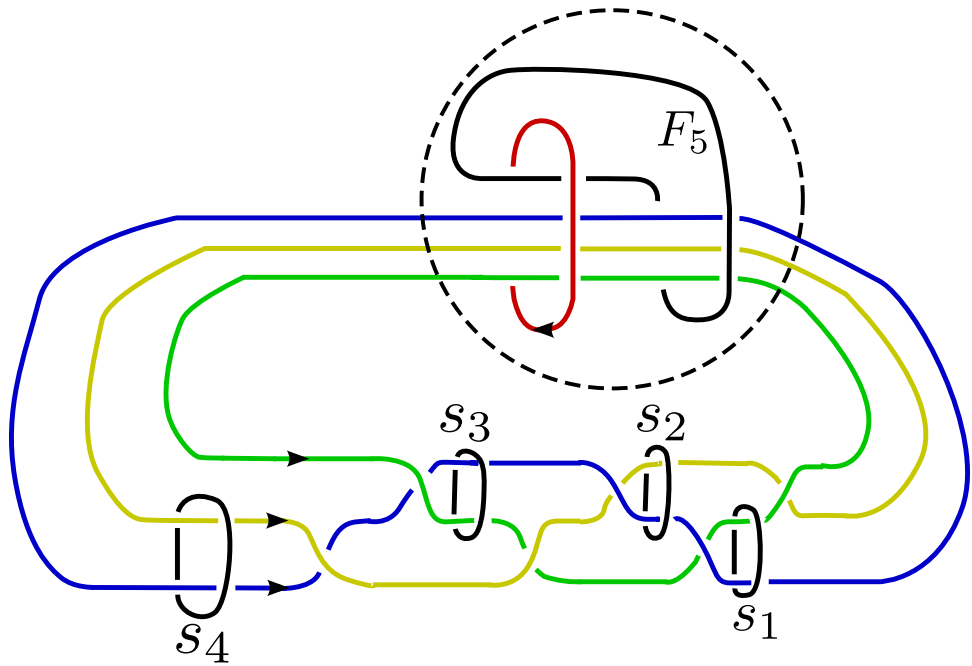}}
\caption{Modified presentation of the bounding manifold as link complement in $S^2\times S^1$.}\label{framedlink4}
\end{figure}

We can now proceed to remove the last framed component, therefore obtaining a presentation of the bounding manifold as the exterior of a link in $S^3$. We use a sequence of local moves, which all take place within the region highlighted by a dashed line in Figure \ref{framedlink4}. We begin by performing a clockwise twist on the component labeled $r$, which changes to $-1$ the framing on $F_5$, unlinking it from all the components except the $r$. A blow-down along $F_5$ allows us to remove this last framed component.
In practice, the removal of the framed component is achieved by a full clockwise twist of the green, blue and yellow components (see Figure \ref{locale}).

\begin{figure}
\centering
\makebox[\textwidth][c]{
\includegraphics[width=\textwidth]{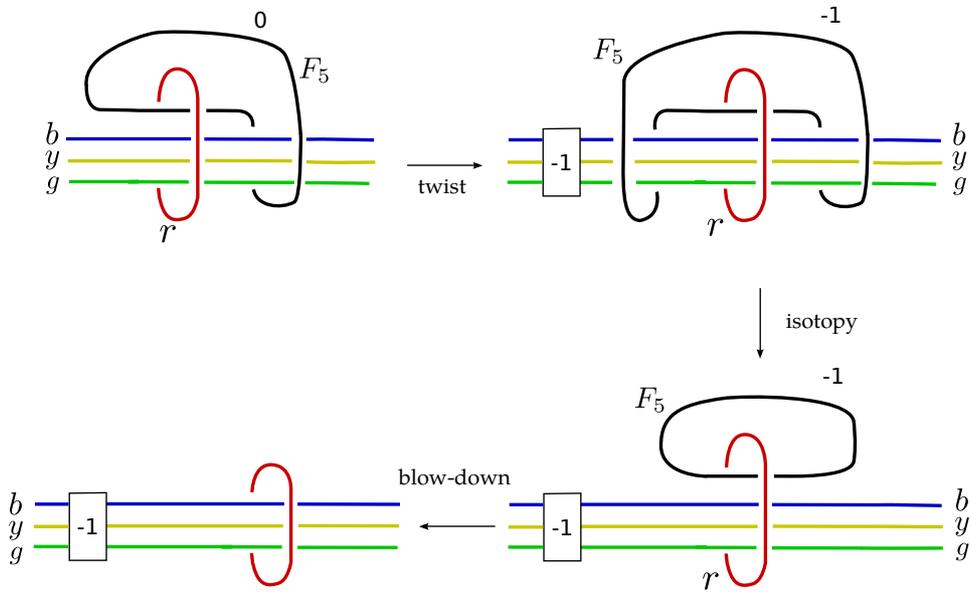}}
\caption{Kirby moves removing the last framed component.}\label{locale}
\end{figure}
The result of this move is shown in Figure \ref{finale}. 
\begin{figure}
\centering
\makebox[\textwidth][c]{
\includegraphics[width=0.8\textwidth]{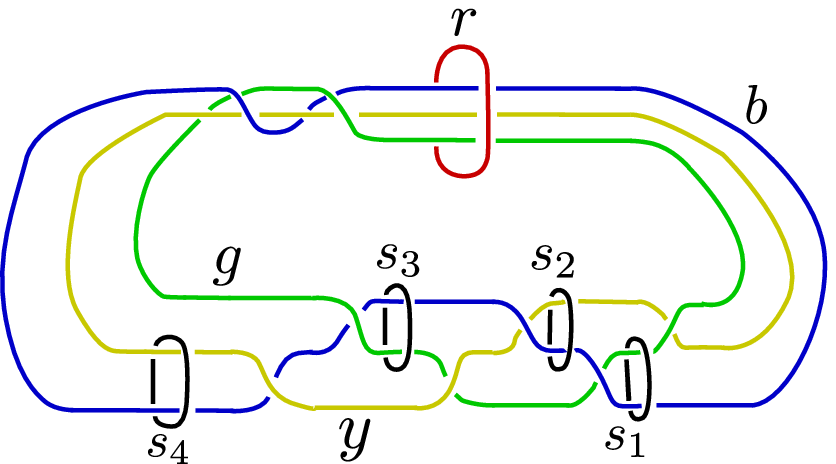}}
\caption{Presentation of the bounding manifold as link complement in $S^3$.}\label{finale}
\end{figure}
As a final step, we apply a counter-clockwise twist along the component labeled $r$. This moves removes the crossings to the left of the component labeled $r$, and simplifies the link complement to the one of Figure \ref{finale2}.

\begin{figure}[ht!]
\centering
\makebox[\textwidth][c]{
\includegraphics[width=\textwidth]{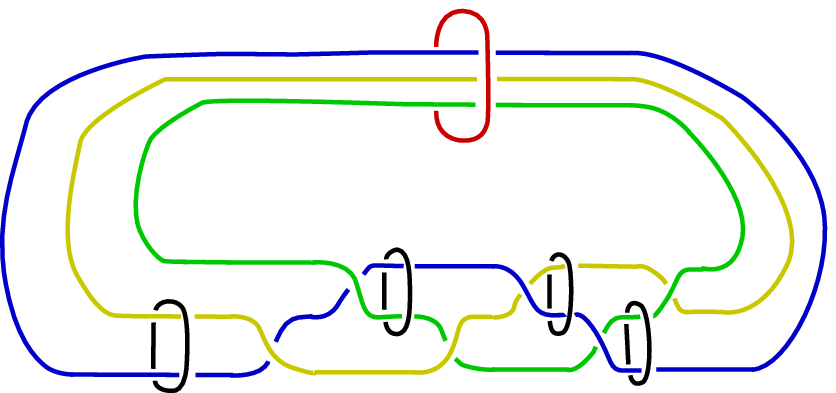}}
\caption{Simplified presentation of the bounding manifold as link complement in $S^3$.}\label{finale2}
\end{figure}


\clearpage



\begin{thebibliography}{}
\bibitem{martelli3} F. Costantino, R. Frigerio, B. Martelli, C. Petronio: {\em Triangulations of $3$-manifolds, hyperbolic
relative handlebodies, and Dehn filling}, Comm. Math. Helv. 82 (2007), 903-934 
\bibitem{gibbons} G. W. Gibbons: {\em Tunneling with a negative cosmological constant}, Nuclear Phys.
B, 472 (1996) 683-708
\bibitem{GPS} M. Gromov, I. Piatetski-Shapiro: {Non-arithmetic groups in Lobachevsky spaces}, Inst. Hautes Etudes Sci.
Publ. Math. 66 (1988), 93-103
\bibitem{martelli} A. Kolpakov, B. Martelli: {\em Hyperbolic $4$-manifolds with one cusp}, Geom. \& Funct. Anal. 23 (2013), 1903-1933
\bibitem{martelli4} A. Kolpakov, B. Martelli, S. Tschanz: {\em Some hyperbolic three-manifolds that bound geometrically}, arXiv:1311.2993
\bibitem{longreid} D. D. Long, A. W. Reid: {\em On the geometric boundaries of hyperbolic $4$-manifolds}, Geometry \& Topology, Volume 4 (2000) 171-178
\bibitem{longreid1} D. D. Long, A. W. Reid: {\em Constructing hyperbolic manifolds which bound geometrically}, Math. Research Lett. 8 (2001), 443-456
B, 472 (1996) 683-708
\bibitem{ratcliffetschanz} J. G. Ratcliffe, S. T. Tschantz: {\em Gravitational instantons of constant curvature},
Classical and Quantum Gravity, 15 (1998) 2613-2627
\bibitem{ratcliffetschanz2} J. G. Ratcliffe, S. T. Tschantz: {\em The volume spectrum of hyperbolic $4$-manifolds},
Experimental J. Math. 9 (2000) 101-125
\bibitem{ratcliffetschanz3} J. G. Ratcliffe, S. T. Tschantz: {\em On the growth of the number of hyperbolic gravitational instantons with respect to volume}, Class. Quantum Grav. 17 (2000) 2999-3007
 
\bibitem{rolfsen} D. Rolfsen: {\em Knots and links}, AMS Chelsea Publishing, American Mathematical Society (2003)  

\bibitem{notes} W. P. Thurston: {\em The Geometry and Topology of $3$-manifolds},
mimeographed notes, Princeton, 1979
531-540
\end{thebibliography}
\end{document}